\documentclass[a4paper,11pt]{article}
\usepackage{amssymb,amsmath,amsthm,amsfonts}
\usepackage{bookmark}
\usepackage{mathrsfs}
\usepackage{multirow}
\usepackage{graphicx}
\usepackage{authblk}
\usepackage{indentfirst}
\usepackage{multicol}
\usepackage{tabu}
\usepackage{url}
\usepackage{fancyhdr}
\usepackage[numbers]{natbib}
\usepackage[all]{xy}
\usepackage{mathtools}
\usepackage[english]{babel}
\usepackage{colortbl}
\usepackage{caption}
\usepackage{hyperref}\usepackage{subfigure}
\usepackage{float}
\usepackage{xcolor}\usepackage{footmisc}\usepackage{booktabs}

\newtheorem{theorem}{Theorem}[section]
\newtheorem{corollary}[theorem]{Corollary}

\newtheorem{lemma}[theorem]{Lemma}
\newtheorem{proposition}[theorem]{Proposition}

\theoremstyle{definition}
\newtheorem{definition}[theorem]{Definition}
\newtheorem{example}[theorem]{Example}
\newtheorem{remark}[theorem]{Remark}

\DeclareMathOperator{\im}{im}

\DeclareMathOperator{\rank}{rank}

\DeclareMathOperator{\magn}{Mag}

\usepackage{xr}
\makeatletter
\newcommand*{\addFileDependency}[1]{
	\typeout{(#1)}
	\@addtofilelist{#1}
	\IfFileExists{#1}{}{\typeout{No file #1.}}
}
\makeatother

\title{Persistent magnitude homology on finite metric spaces}

\author[1]{Wanying Bi}
\author[2]{Hongsong Feng}
\author[3]{Jingyan Li}
\author[3]{Jie Wu \thanks{Corresponding author: wujie@bimsa.cn}}
\affil[1]{School of Mathematical Sciences, Chongqing University of Technology, Chongqing 400054, China}
\affil[2]{Department of Mathematics and Statistics, University of North Carolina at Charlotte, Charlotte, NC 28223, USA}
\affil[3]{Yanqi Lake Beijing Institute of Mathematical Sciences and Applications, Beijing 101408, China}

\makeatletter
\renewcommand*{\@fnsymbol}[1]{\ensuremath{\ifcase#1\or \dagger\or *\or *\or
		\mathsection\or \else\@ctrerr\fi}}
\makeatother
\date{}

\begin{document}
		\maketitle
		
		\paragraph{Abstract}

Magnitude homology is an emerging approach that captures the intrinsic topological and geometric features of metric spaces. It offers a distinct theoretical lens for interpreting structural information within topological and geometric data analysis. This work introduces persistent magnitude homology, an extension of magnitude homology that captures multi-scale geometric and topological features of metric spaces. We construct the category of finite metric spaces with isometric embeddings and show that magnitude homology defines a functor to the category of abelian groups, naturally leading to the definition of persistent magnitude homology. We also introduce weighted persistent modules and weighted barcodes to offer both an algebraic and visual description of persistent magnitude homology. Additionally, we present an isometry theorem that relates interleaving distances and bottleneck distances, and establish stability results for persistent magnitude homology and magnitude profile. These results establish the stability of magnitude-based descriptors, bridging the gap between theory and practical application.



\tableofcontents

\section{Introduction}\label{section:introduction}

Magnitude, introduced by Tom Leinster in 2013, is a numerical invariant that measures the ``size'' or ``complexity'' of metric spaces and enriched categories, generalizing the Euler characteristic and set cardinality~\cite{leinster2013magnitude}. It can be used to capture geometric and combinatorial properties like diversity, capacities, intrinsic dimensions, and asymptotic behaviors in Euclidean spaces~\cite{leinster2017magnitude,leinster2021magnitude,leinster2013asymptotic,meckes2015magnitude}. Later, magnitude was extended to graphs, quantifying their structural features using the shortest path metric~\cite{leinster2019magnitude}.

Hepworth and Willerton advanced the concept of magnitude by developing magnitude homology, a graded algebraic structure that categorifies magnitude and recovers it as its Euler characteristic~\cite{hepworth2017categorifying}. Leinster and Shulman introduced magnitude homology for enriched categories and metric spaces in 2021 \cite{leinster2021magnitude}, and then Hepworth extended this concept to magnitude cohomology \cite{hepworth2022magnitude}. Recent developments in magnitude homology have explored its applications to metric spaces, order complexes, and graphs, as well as its connections to torsion, girth, and phase transitions \cite{asao2024girth,kaneta2021magnitude,sazdanovic2021torsion}.
The study of magnitude homology and its relationship to path homology was further explored in 2023 \cite{asao2023magnitude}. In \cite{bi2024magnitude}, the authors introduced the magnitude homology of hypergraphs and the K\"{u}nneth formula for their magnitude homology.

When describing data, magnitude is particularly suited for capturing structural objects such as molecules, proteins, networks, and so on \cite{bi2025topological, bi2024persistent}. However, as it reflects a global characteristic, its application is inherently limited. Persistent homology is a relatively new approach that transforms coarse topological features into more detailed multi-scale topological features, enabling effective data analysis \cite{edelsbrunner2008persistent,zomorodian2004computing}. The study of the persistence of magnitude homology has remained a challenging problem. Existing research on the persistence of magnitude and its connection to persistent homology has been conducted \cite{govc2021persistent,otter2022magnitude}. However, this notion of persistent magnitude does not represent true topological or geometric persistence, as it does not involve the filtration of the space or the data objects themselves.

In this work, we focus on the concept of persistent magnitude homology. First, we construct the category $\mathbf{FinMet}_{iso}$ of finite metric spaces with morphisms given by isometric embeddings. We show that (see Theorem \ref{theorem:functor})
\begin{theorem}
For a fixed non-negative integer $k \geq 0$ and length parameter $l \geq 0$, the magnitude homology $\mathrm{MH}_{k, l}(-): \mathbf{FinMet}_{iso} \to \mathbf{Ab}$ is a functor.
\end{theorem}
This result naturally leads to the definition of persistent magnitude homology. Based on persistent magnitude homology, we define the corresponding persistent magnitude and establish its algebraic relationship with the classical magnitude.

On the other hand, we note that magnitude homology has two indices, one corresponding to the homological dimension and the other to the length of the homology generators. Motivated by this structure, we introduce the concepts of weighted persistence modules and weighted barcodes. These notions provide a robust algebraic characterization and a corresponding visual representation for persistent magnitude homology. A central result of our study is the isometry theorem (Theorem \ref{theorem:isometry}), which establishes the theoretical equivalence between the bottleneck distance of weighted barcodes and the interleaving distance of weighted persistence modules:
\begin{theorem}
Let $\mathcal{V}$ and $\mathcal{W}$ be two pointwise finite-dimensional weighted persistence modules with associated weighted barcodes $\mathcal{B}_{\mathcal{V}}$ and $\mathcal{B}_{\mathcal{W}}$, respectively. Then
$$ d_B(\mathcal{B}_{\mathcal{V}}, \mathcal{B}_{\mathcal{W}}) = d_I(\mathcal{V}, \mathcal{W}). $$
\end{theorem}

Finally, for a point set $X$, we construct a subset $N_r(z)=\{x\in X \mid \|x-z\|\leq r\}$ of $X$ centered at a given point $z$, with distances to $z$ bounded by $r$. This construction yields a persistent finite metric space particularly suited for practical multiscale analysis. Specifically, for a non-decreasing function $f:\mathbb{R}\to \mathbb{R}^{+}$, we denote its generalized inverse as 
\[
f^{-1}(y) = \inf \{ x \in \mathbb{R} \mid f(x) \geq y \}.
\]
Then for a given point set $X$ and a fixed point $z$ in the space, we obtain a persistent finite metric space
\[
N(z,f) = N_{f^{-1}(-)}(z): (\mathbb{R}, \leq) \longrightarrow \mathbf{FinMet}_{iso}.
\]
This stability is formalized in the following theorem (see Theorem \ref{theorem:stability_composition}):
\begin{theorem}\label{theorem:stability_composition}
Let $f,g: \mathbb{R} \to \mathbb{R}^{+}$ be two non-decreasing functions. Assume $f$ is subadditive. We have
\[
  d_B(\mathcal{B}(N(z,f)), \mathcal{B}(N(z',g))) \leq \| f - g \|_\infty + f(\| z - z' \|).
\]
\end{theorem}

The stability of magnitude breaks down in dense configurations, as the invariant becomes singular when the pairwise distance between points vanishes. To address this, we restrict our analysis to the thick configuration space
\[
  \mathcal{C}_n^{\delta}(\mathbb{R}^d) = \left\{ (x_1, \dots, x_n) \in (\mathbb{R}^d)^n \mid \|x_i - x_j\| \ge \delta \text{ for } i \neq j \right\},
\]
which ensures that any two distinct points in the set maintain a minimum separation distance $\delta$. Without loss of generality, we assume the point sets are centered at the origin and contained within a disk of radius $L$. Under these geometric constraints, magnitude perturbations become controllable.

Specifically, we investigate the magnitude profile of a persistent metric space, defined as $\mathrm{Mag}_X(r) := \mathrm{Mag}(N_r(X))$. By employing the $L^1([0,L])$ metric to measure the distance between profiles and the $\infty$-Wasserstein metric between point sets, we establish the following stability result (Theorem \ref{theorem:stability}):
\begin{theorem}
Let $X, Y \in \mathcal{C}_n^{\delta}(\mathbb{R}^d)$ be two finite point sets contained within a disk of radius $L$. Then
\[
  d_{L^1}(\mathrm{Mag}_X, \mathrm{Mag}_Y) \leq K_{n,d,L,\delta} \cdot d_{W,\infty}(X, Y),
\]
where $K_{n,d,L,\delta}$ is a constant depending on the cardinality $n$, dimension $d$, radius $L$, and separation threshold $\delta$.
\end{theorem}
This theorem implies that in practical applications of magnitude theory, a certain degree of spatial dispersion is required to ensure the robustness of the extracted geometric features.

The paper is organized as follows. In the next section, we introduce persistent magnitude homology. Section \ref{section:representation} explores weighted persistence modules and weighted barcodes. In the final section, we present the stability theorems of magnitude invariants.

\section{Persistent magnitude homology}

Magnitude is a topological invariant that encodes the size or scale of finite metric spaces, or more specifically, the size of finite point sets and finite graphs.
For a finite point set $X$, the magnitude of $X$ is a measure that captures the intrinsic geometric properties of $X$, such as the pairwise distances between points and the global structure of the set.

\subsection{Magnitude of finite metric spaces}

A \textit{finite metric space} is a pair $(X, d)$, where $X$ is a finite set and $d: X \times X \to \mathbb{R}_{\geq 0}$ is a metric. These spaces are finite in cardinality and provide a discrete, compact structure.

Let $(X, d)$ be a finite metric space, where $X = \{x_1, x_2, \dots, x_n\}$ is a set of $n$ points. Define the similarity matrix $Z\in \mathbb{R}^{n \times n}$ with entries
\[
Z_{ij} = e^{-d(x_i, x_j)},
\]
where $d(x_i, x_j)$ is the distance between points $x_i$ and $x_j$. Suppose that $Z$ is invertible. The \textit{magnitude} of $X$ is given by
\[
\magn(X) = \mathbf{1}^{T} Z^{-1} \mathbf{1} = \sum_{i=1}^n \sum_{j=1}^n (Z^{-1})_{ij},
\]
where $\mathbf{1}$ is the all-ones vector in $\mathbb{R}^n$, and $\mathbf{1}^{T}$ is the transpose of $\mathbf{1}$.

Equivalently, if there exists a \textit{weighting} $w: X \to \mathbb{R}$ such that
\[
\sum_{y \in X} e^{-d(x, y)} w(y) = 1
\]
for all $x \in X$, then the magnitude is
\[
\magn(X) = \sum_{x \in X} w(x).
\]

The construction of magnitude homology begins with the chain groups $\mathrm{MC}_{k,l}(X)$, defined as the free abelian group generated by $(k+1)$-tuples $(x_0, x_1, \dots, x_k) \in X^{k+1}$ satisfying
\[
d(x_0, x_1) + d(x_1, x_2) + \cdots + d(x_{k-1}, x_k) = l
\]
for some $l \in \mathbb{R}_{\geq 0}$, and the non-degeneracy condition $x_i \neq x_{i+1}$ for all $i = 0, \dots, k-1$. These tuples represent ``paths'' in the metric space with a fixed total length, excluding consecutive repeated points

For an element $\sigma=(x_0, x_1, \dots, x_k)$, we define
\begin{equation*}
  \ell(\sigma)=d(x_0, x_1) + d(x_1, x_2) + \cdots + d(x_{k-1}, x_k).
\end{equation*}
Then the boundary operator is given by
\[
  d_{k,l}: \mathrm{MC}_{k,l}(X) \to \mathrm{MC}_{k-1,l}(X), \quad d_{k,l} = \sum_{i=0}^k (-1)^i \partial_{i,l},
\]
where
\begin{equation*}
  \partial_{i,l}(x_0, \dots, x_k) = \left\{
                                    \begin{array}{ll}
                                      (x_0, \dots, \widehat{x}_i, \dots, x_k), & \hbox{$\ell(x_0, \dots, \widehat{x}_i, \dots, x_k)=l$;} \\
                                      0, & \hbox{otherwise.}
                                    \end{array}
                                  \right.
\end{equation*}
Here, $\widehat{x}_i$ denotes the omission of the $i$-th point.

The operator satisfies $d_{k-1,l} \circ d_{k,l} = 0$, forming a chain complex $(\mathrm{MC}_{\ast, \ast}(X), d_{\ast, \ast})$.

\begin{definition}
The \textit{magnitude homology} is defined by
\[
\mathrm{MH}_{k,l}(X) = \frac{\ker (d_{k,l}: \mathrm{MC}_{k,l}(X) \to \mathrm{MC}_{k-1,l}(X))}{\mathrm{im} (d_{k+1,l}: \mathrm{MC}_{k+1,l}(X) \to \mathrm{MC}_{k,l}(X))}.
\]
\end{definition}
The magnitude homology group is bigraded by homological degree $k$ and length $l$. The magnitude $\magn(X)$ is recovered by the magnitude homology
\begin{equation*}
   \magn(X) = \sum_{l \geq 0} \chi \left(\mathrm{MH}_{\ast,l}(X)\right) \, e^{-l},
\end{equation*}
where
\begin{equation*}
  \chi \left(\mathrm{MH}_{\ast,l}(X)\right) = \sum_{k=0}^\infty (-1)^k \,\rank \big(\mathrm{MH}_{k,l}(X)\big).
\end{equation*}
The $0$-dimensional magnitude homology reflects the discrete points in a finite metric space $X$. The $1$-dimensional magnitude homology, denoted $\mathrm{MH}_{1,l}$, characterizes the cycles of length $l$ that cannot be filled. The $2$-dimensional magnitude homology, $\mathrm{MH}_{2,l}$, describes the cycles formed by paths corresponding to triples of points $(x_0, x_1, x_2)$ of length $l$.

Magnitude homology provides a richer invariant than magnitude alone, capturing higher-order geometric and topological structure.

\subsection{Persistent magnitude homology}

Consider the category $\mathbf{FinMet}_{iso}$ of finite metric spaces and isometric embeddings between them. The objects of this category are finite metric spaces $(X, d)$, where $X$ is a finite set and $d: X \times X \to \mathbb{R}_{\geq 0}$ is a metric on $X$. The morphisms in this category are isometric embeddings. Specifically, an isometric embedding is an injective map $f: (X, d_X) \to (Y, d_Y)$ between finite metric spaces such that for all $x, x' \in X$, the distance between $f(x)$ and $f(x')$ in $Y$ is the same as the distance between $x$ and $x'$ in $X$, i.e., $d_Y(f(x), f(x')) = d_X(x, x')$.

\begin{theorem}\label{theorem:functor}
For a fixed non-negative integer $k \geq 0$ and length parameter $l \geq 0$, the magnitude homology $\mathrm{MH}_{k, l}(-): \mathbf{FinMet}_{iso} \to \mathbf{Ab}$ is a functor from the category of finite metric spaces with isometric embeddings to the category of abelian groups.
\end{theorem}

\begin{proof}
To prove that $\mathrm{MH}_{k,l}(-): \mathbf{FinMet}_{iso} \to \mathbf{Ab}$ is a functor, we verify that it satisfies the necessary properties of functors: it assigns a group homomorphism to each isometric embedding, preserves identity morphisms, and preserves composition of morphisms.

First, let $f: (X, d_X) \to (Y, d_Y)$ be an isometric embedding. We define a chain map
\[
f_*: \mathrm{MC}_{k,l}(X) \to \mathrm{MC}_{k,l}(Y), \quad f_\ast(x_0, \dots, x_k) = (f(x_0), \dots, f(x_k)).
\]
We now check that $f_\ast$ is well-defined. If $(x_0, \dots, x_k) \in \mathrm{MC}_{k,l}(X)$, then
\[
d_X(x_0, x_1) + \cdots + d_X(x_{k-1}, x_k) = l, \quad x_i \neq x_{i+1}.
\]
Since $f$ is an isometric embedding, we have
\[
d_Y(f(x_0), f(x_1)) + \cdots + d_Y(f(x_{k-1}), f(x_k)) = d_X(x_0, x_1) + \cdots + d_X(x_{k-1}, x_k) = l,
\]
and since $f$ is injective, we also have $f(x_i) \neq f(x_{i+1})$ whenever $x_i \neq x_{i+1}$. Thus, we obtain
\[
(f(x_0), \dots, f(x_k)) \in \mathrm{MC}_{k,l}(Y),
\]
so $f_\ast$ is well-defined.

To show that $f_\ast$ is a chain map, we must verify that it commutes with the boundary operator. By a direct calculation, we have
\[
\partial_{k,l}(f_\ast(x_0, \dots, x_k)) = \partial_{k,l}(f(x_0), \dots, f(x_k)) = \sum_{i=0}^k (-1)^i (f(x_0), \dots, \widehat{f(x_i)}, \dots, f(x_k)).
\]
On the other hand, we obtain
\begin{align*}
  f_\ast(\partial_{k,l}(x_0, \dots, x_k)) = & f_\ast\left( \sum_{i=0}^k (-1)^i (x_0, \dots, \widehat{x_i}, \dots, x_k) \right) \\
    = &\sum_{i=0}^k (-1)^i (f(x_0), \dots, \widehat{f(x_i)}, \dots, f(x_k)).
\end{align*}
Since both expressions are equal, we have $\partial_{k,l} \circ f_\ast = f_\ast \circ \partial_{k,l}$, so $f_\ast$ is indeed a chain map.

This induces a homomorphism on homology
\[
\mathrm{MH}_{k, l}(f): \mathrm{MH}_{k, l}(X) \to \mathrm{MH}_{k, l}(Y),
\]
defined by $\mathrm{MH}_{k, l}(f)([z]) = [f_\ast(z)]$, where $[z] \in \mathrm{MH}_{k, l}(X)$ is the homology class of a cycle $z \in \ker \partial_{k,l}$. This homomorphism is well-defined.

For isometric embeddings $f: (X, d_X) \to (Y, d_Y)$ and $g: (Y, d_Y) \to (Z, d_Z)$, the composition $g \circ f$ induces the chain map $(g \circ f)_\ast = g_\ast \circ f_\ast$, and on homology, we have $\mathrm{MH}_{k, l}(g \circ f) = \mathrm{MH}_{k, l}(g) \circ \mathrm{MH}_{k, l}(f)$, which follows directly from the functoriality of the homology.

Additionally, one can verify that $\mathrm{MH}_{k, l}(-)$ preserves the identity homomorphism.

In summary, the magnitude homology functor $\mathrm{MH}_{k, l}(-): \mathbf{FinMet}_{iso} \to \mathbf{Ab}$ is a functor.
\end{proof}

From the above theorem, it follows that a finite point set in Euclidean space naturally forms a finite metric space. Therefore, we have the following corollary. Let $\mathbf{FinPts}^{\hookrightarrow}$ be the category of finite point sets in Euclidean space, where the objects are finite sets of points, and the morphisms are geometric embeddings of point sets.

\begin{corollary}
For a fixed non-negative integer $k \geq 0$ and length parameter $l \geq 0$, the magnitude homology functor $\mathrm{MH}_{k,l}(-): \mathbf{FinPts}^{\hookrightarrow} \to \mathbf{Ab}$ is a functor.
\end{corollary}

Let $X \subseteq Y$ be an inclusion of point sets in a metric space. Then, we can treat both $X$ and $Y$ as finite metric spaces, with the metric inherited from the underlying metric space. In this setting, we have the geometric embedding of finite metric spaces $f_{X,Y}: X \to Y$.

Consider the category $(\mathbb{R}, \leq)$ where the objects are real numbers and the morphisms are given by the partial order, i.e., $a \to b$ if and only if $a \leq b$.
\begin{definition}
 A \textit{persistence finite metric space} is a functor $\mathcal{S}:(\mathbb{R}, \leq) \to \mathbf{FinMet}_{iso}$ from the category $(\mathbb{R}, \leq)$ to the category of finite metric spaces.
\end{definition}

In particular, a functor $\mathcal{S}: (\mathbb{R}, \leq) \to \mathbf{FinPts}^{\hookrightarrow}$ can also be regarded as a persistence finite metric space.

Let $\mathcal{S}:(\mathbb{R}, \leq) \to \mathbf{FinMet}_{iso}$ be a persistence finite metric space. For any real numbers $a\leq b$, we have a geometric embedding $\mathcal{S}_{a}\to \mathcal{S}_{b}$ of finite metric spaces. This leads to the morphism of magnitude homology
\begin{equation*}
  \mathrm{MH}_{k,l}(-): \mathrm{MH}_{k,l}(\mathcal{S}_{a})\to \mathrm{MH}_{k,l}(\mathcal{S}_{b}).
\end{equation*}

\begin{definition}
The $(a,b)$-\textit{persistent magnitude homology} of $\mathcal{S}$ is defined by
\begin{equation*}
  \mathrm{MH}_{k,l}^{a,b}(\mathcal{S}) = \im \left(\mathrm{MH}_{k,l}(\mathcal{S}_{a})\to \mathrm{MH}_{k,l}(\mathcal{S}_{b})\right),\quad k\geq 0,l\geq 0.
\end{equation*}
\end{definition}
Given that the magnitude homology has two subscripts, the first representing the dimension and the second representing the cycle length, it follows that persistent magnitude homology possesses richer geometric and topological features.

The rank of $\mathrm{MH}_{k,l}^{a,b}(\mathcal{S})$ is called the $(a,b)$-persistent magnitude Betti number, denoted by $\beta_{k,l}^{a,b}(\mathcal{S})=\rank \left(\mathrm{MH}_{k,l}^{a,b}(\mathcal{S})\right)$. Here, $\rank \left(\mathrm{MH}_{k,l}^{a,b}(\mathcal{S})\right)$ is the rank of the free abelian group $\mathrm{MH}_{k,l}^{a,b}(\mathcal{S})$.

Analogous to classical persistent homology, one may study the barcode and the persistence diagram associated with persistent magnitude homology. Moreover, associated results such as a persistence module decomposition theorem can also be developed in the context of magnitude homology.

To ensure the structure theorem for persistence modules holds over a field, we consider the magnitude homology with coefficients in a field $\mathbb{K}$. Let us denote
\begin{equation*}
  \mathbf{MH} = \bigoplus\limits_{a} \mathrm{MH}_{\ast,\ast}(\mathcal{S}_{a}; \mathbb{K}).
\end{equation*}
If the filtration parameter is an integer, that is, for $\mathcal{S}: (\mathbb{Z}, \leq) \to \mathbf{FinMet}_{iso}$, we have a map
\begin{equation*}
  t : \mathrm{MH}_{\ast,\ast}(\mathcal{S}_{a}; \mathbb{K}) \to \mathrm{MH}_{\ast,\ast}(\mathcal{S}_{a+1}; \mathbb{K}), \quad a \in \mathbb{Z}.
\end{equation*}
This induces a linear map
\begin{equation*}
  t : \mathbf{MH} \to \mathbf{MH}
\end{equation*}
of degree 1. Thus, $\mathbf{MH}$ is a $\mathbb{K}[t]$-module, providing a persistence module structure for $\mathbf{MH}$. Consequently, we obtain a decomposition theorem for persistence modules.

\begin{theorem}
Let $\mathbf{MH}$ be a finitely generated $\mathbb{K}[t]$-module. Then we have a decomposition
\begin{equation*}
  \mathbf{MH} \cong \left( \bigoplus\limits_{i} e_{b_i}^{i}(l_i) \cdot \mathbb{K}[t] \right) \oplus \left( \bigoplus\limits_{j} \frac{\tilde{e}_{c_j}^{j}(m_j) \cdot \mathbb{K}[t]}{(t^{d_j})} \right),
\end{equation*}
where $e_{b_i}^{i}(l_i)$ is the free generator of $\bigoplus\limits_{a} \mathrm{MH}_{\ast, l_i}(\mathcal{S}_a; \mathbb{K})$, with birth time $b_i$ and no death time, and $\tilde{e}_{c_j}^{j}(m_j)$ is the torsion generator of $\bigoplus\limits_{a} \mathrm{MH}_{\ast, m_j}(\mathcal{S}_a; \mathbb{K})$, with birth time $c_j$ and death time $c_j + d_j$.
\end{theorem}

The decomposition above is essentially the same as the usual decomposition of persistence modules, with the only difference being that the generators not only have information about their birth and death times, but also carry weight information. This weight comes from the second index of the magnitude homology, namely the length. In the subsequent sections, we will introduce weighted persistence modules to capture this phenomenon.

\subsection{Persistent magnitude and its categorification}

Let $\mathcal{S}:(\mathbb{R}, \leq) \to \mathbf{FinMet}_{iso}$ be a persistence finite metric space. For any real numbers $a \leq b$, let us denote
\begin{equation*}
   \magn^{a,b}(\mathcal{S}) = \sum_{l \geq 0} \chi \left(\mathrm{MH}_{\ast,l}^{a,b}(\mathcal{S})\right) \, e^{-l},
\end{equation*}
where
\begin{equation*}
  \chi \left(\mathrm{MH}_{\ast,l}^{a,b}(\mathcal{S})\right) = \sum_{k=0}^\infty (-1)^k \,\rank \big(\mathrm{MH}_{k,l}^{a,b}(\mathcal{S})\big)
\end{equation*}
denotes the Euler characteristic of the persistent magnitude homology at scale $l$, and $\rank(\mathrm{MH}_{k,l}^{a,b}(\mathcal{S}))$ is the rank of the free abelian group $\mathrm{MH}_{k,l}^{a,b}(\mathcal{S})$. We call $\magn^{a,b}(\mathcal{S})$ the $(a,b)$-\textit{persistent magnitude} of $\mathcal{S}$.

\begin{proposition}
Let $\mathcal{S}:(\mathbb{R}, \leq) \to \mathbf{FinMet}_{iso}$ be a persistence finite metric space. For $a \leq b$, the $(a,b)$-persistent magnitude satisfies
\[
\magn^{a,b}(\mathcal{S}) = \magn(\mathcal{S}_a) - \sum_{l \geq 0} \chi \left( \ker\left(\mathrm{MH}_{\ast,l}(\mathcal{S}_a) \to \mathrm{MH}_{\ast,l}(\mathcal{S}_b)\right) \right) e^{-l}.
\]
In particular, $\magn^{a,b}(\mathcal{S}) = \magn(\mathcal{S}_a)$ if and only if the induced homomorphism $\mathrm{MH}_{k,l}(\mathcal{S}_a) \to \mathrm{MH}_{k,l}(\mathcal{S}_b)$ is injective for all $k, l \geq 0$.
\end{proposition}

\begin{proof}
Let $f_{a,b}: \mathcal{S}_a \to \mathcal{S}_b$ be the associated map. By definition, the $(a,b)$-persistent magnitude homology is the image of the induced homomorphism:
\[
\mathrm{MH}_{k,l}^{a,b}(\mathcal{S}) = \im \left( \mathrm{MH}_{k,l}(f_{a,b}): \mathrm{MH}_{k,l}(\mathcal{S}_a) \to \mathrm{MH}_{k,l}(\mathcal{S}_b) \right).
\]
By the Rank-Nullity Theorem for abelian groups, we have the short exact sequence:
\[
0 \to \ker\left(\mathrm{MH}_{k,l}(f_{a,b})\right) \to \mathrm{MH}_{k,l}(\mathcal{S}_a) \to \mathrm{MH}_{k,l}^{a,b}(\mathcal{S}) \to 0,
\]
which yields the rank relation:
\[
\rank \big( \mathrm{MH}_{k,l}^{a,b}(\mathcal{S}) \big) = \rank \big( \mathrm{MH}_{k,l}(\mathcal{S}_a) \big) - \rank \left( \ker\left(\mathrm{MH}_{k,l}(f_{a,b})\right) \right).
\]
Multiplying by $(-1)^k e^{-l}$ and summing over all $k \ge 0$ and $l \ge 0$, we obtain:
\begin{align*}
  \magn^{a,b}(\mathcal{S})
    &= \sum_{l \geq 0} \sum_{k=0}^\infty (-1)^k \, \rank \big( \mathrm{MH}_{k,l}^{a,b}(\mathcal{S}) \big) \, e^{-l} \\
    &= \sum_{l \geq 0} \sum_{k=0}^\infty (-1)^k \left( \rank \big( \mathrm{MH}_{k,l}(\mathcal{S}_a) \big) - \rank \left( \ker\left(\mathrm{MH}_{k,l}(f_{a,b})\right) \right) \right) e^{-l} \\
    &= \magn(\mathcal{S}_a) - \sum_{l \geq 0} \chi \left( \ker\left(\mathrm{MH}_{\ast,l}(f_{a,b})\right) \right) e^{-l}.
\end{align*}
The equality $\magn^{a,b}(\mathcal{S}) = \magn(\mathcal{S}_a)$ holds if and only if the correction term is zero for all $l \ge 0$, which requires $\ker(\mathrm{MH}_{k,l}(f_{a,b})) = 0$ for all $k, l \ge 0$. This is equivalent to $f_{a,b}$ being injective on homology.
\end{proof}

The above result indicates that the persistent magnitude $\magn^{a,b}(\mathcal{S})$ is influenced by the death of homology classes in $[a, b]$. If no generators die, meaning the induced map $\mathrm{MH}_{k,l}(\mathcal{S}_a) \to \mathrm{MH}_{k,l}(\mathcal{S}_b)$ is injective for all $k, l \geq 0$, then the persistent magnitude remains unchanged (i.e., $\magn^{a,b}(\mathcal{S}) = \magn(\mathcal{S}_a)$).

\section{Representation of persistent magnitude homology}\label{section:representation}

In this section, we introduce a representation of persistent magnitude homology. Unlike classical persistent homology, we use weighted persistence modules and weighted barcodes to provide both an algebraic representation and a geometric visualization of persistent magnitude homology.
From now on, all vector spaces considered are assumed to be finite-dimensional, and accordingly, the corresponding persistence modules are pointwise finite-dimensional.

\subsection{Weighted persistence module}

\begin{definition}
A \textit{weighted vector space} is a vector space $V$ graded over the non-negative real numbers, denoted as
\[
V = \bigoplus_{w \in \mathbb{R}_{\geq 0}} V_w,
\]
where $V_w$ is the homogeneous component of weight $w$. 
\end{definition}

We assume that the weighted vector spaces considered are of finite dimension, meaning $V_w = 0$ for all but finitely many $w$.

Let $\mathbf{Vec}^{\ast}_{\mathbb{K}}$ be the category of weighted vector spaces over a field $\mathbb{K}$, with structure defined as follows:
\begin{itemize}
  \item Objects are weighted vector spaces $V = \bigoplus_{w \in \mathbb{R}_{\geq 0}} V_w$.
  \item A morphism $f: V \to V'$ is a $\mathbb{K}$-linear map that preserves the weight, i.e., $f(V_w) \subseteq V'_w$ for all $w \in \mathbb{R}_{\geq 0}$.
  \item The composition of morphisms is the usual composition of $\mathbb{K}$-linear maps.
  \item The identity morphism of an object $V$ is the identity map $\mathrm{id}_V: V \to V$.
\end{itemize}

\begin{example}
Magnitude homology naturally carries a weighted vector space structure. Indeed, for any finite metric space $X$, the magnitude homology $\mathrm{MH}_{\ast,\ast}(X) = \bigoplus_{l \ge 0} \mathrm{MH}_{\ast, l}(X)$ can be viewed as a weighted vector space, where the weight is given by $l \ge 0$.
\end{example}

\begin{definition}
A \textit{weighted persistence module} is a functor $\mathcal{V}: (\mathbb{R},\leq)\to \mathbf{Vec}^{\ast}_{\mathbb{K}}$.
\end{definition}

The magnitude homology can be viewed as a functor $\mathrm{MH}_{\ast,\ast}(-): \mathbf{FinMet}_{iso} \to \mathbf{Vec}^{\ast}_{\mathbb{K}}$, mapping into the category of weighted vector spaces, where the second subscript corresponds to the weight. Then, for a persistence finite metric space $\mathcal{S}: (\mathbb{R}, \leq) \to \mathbf{FinMet}_{iso}$, we obtain a weighted persistence module $\mathrm{MH}_{\ast,\ast}(\mathcal{S}): (\mathbb{R}, \leq) \to \mathbf{Vec}^{\ast}_{\mathbb{K}}$.

\begin{definition}
The \textit{weighted barcode} of a persistence finite metric space $\mathcal{S}$ is defined as the multiset of triples $([b, d], w)$. Here, $b$ and $d$ correspond to the birth and death times of a topological feature in the filtration of $\mathcal{S}$, respectively, while the weight $w \in \mathbb{R}_{\geq 0}$ represents the magnitude length level of the feature. The persistence of each topological feature is given by the length of the bar $d - b$, while its level is indicated by the weight $w$. Formally, it is the collection
\[
\mathcal{B} = \left\{ ([b, d], w) \mid b \text{ is the birth time}, d \text{ is the death time}, w \text{ is the weight} \right\}.
\]
\end{definition}

\begin{definition}
The \textit{weighted persistence diagram} is formally given by the multiset
\[
\mathcal{D} = \left\{ (b, d, w) \mid b \text{ is the birth time}, d \text{ is the death time}, w \text{ is the weight} \right\}.
\]
\end{definition}

\begin{example}
For a persistence finite metric space $\mathcal{S}: (\mathbb{R}, \leq) \to \mathbf{FinMet}_{iso}$, the functor $\mathrm{MH}_{\ast,\ast}(\mathcal{S}): (\mathbb{R}, \leq) \to \mathbf{Vec}^{\ast}_{\mathbb{K}}$ is a weighted persistence module.
\end{example}

\begin{definition}
The \textit{bottleneck distance} between two weighted barcodes $\mathcal{B}_1$ and $\mathcal{B}_2$ is defined as
\[
d_{B}(\mathcal{B}_1, \mathcal{B}_2) = \inf_{\gamma} \sup_{\substack{([b_1, d_1], w) \in \mathcal{B}_1 \\ \gamma(([b_1, d_1], w)) = ([b_2, d_2], w)}} \max\left( |b_1 - b_2|, |d_1 - d_2| \right).
\]
Here, the $\inf$ runs over all matchings $\gamma$ that preserve the weights (i.e., matchings only occur between bars with the same weight $w$).
\end{definition}

\begin{remark}
In computing the bottleneck distance, unmatched bars in weighted barcodes $\mathcal{B}_1$ and $\mathcal{B}_2$ (for a given weight $w$) are paired with diagonal points. An unmatched bar $([b, d], w)$ is projected to the diagonal point $([x, x], w)$ where $x = (b + d)/2$, yielding a distance of $(d - b)/2$. This ensures a perfect matching and keeps the distance finite.
\end{remark}

\subsection{Isometry theorem of weighted persistence modules}

We now review the concept of the interleaving distance, which provides an algebraic characterization for the stability theorem of persistence modules \cite{bubenik2014categorification, chazal2009proximity}.

Consider the translation functor $T_x: (\mathbb{R}, \leq) \to (\mathbb{R}, \leq)$ given by $T_x(a) = a + x$ for all $a, x \in \mathbb{R}$. Note that an object in the indexed category $ (\mathbf{Vec}^\ast_{\mathbb{K}})^{\mathbb{R}}$, which is the category of functors from $\mathbb{R}$ to $\mathbf{Vec}^\ast_{\mathbb{K}}$, corresponds to a weighted persistence module.

The translation functor $T_x$ induces an endofunctor
\begin{equation*}
  \Sigma^x: (\mathbf{Vec}^\ast_{\mathbb{K}})^{\mathbb{R}} \to (\mathbf{Vec}^\ast_{\mathbb{K}})^{\mathbb{R}},\quad (\Sigma^x \mathcal{V})(a) = \mathcal{V}(a + x)
\end{equation*}
on the category of weighted persistence modules. For $x \geq 0$, the endofunctor $\Sigma^x$ induces a natural transformation from $\mathcal{V}$ to $\Sigma^x \mathcal{V}$ given by the structure maps of the persistence module. For convenience, we simply denote this structural natural transformation by $\Sigma^x$ when the context is clear.

\begin{definition}
Two weighted persistence modules $\mathcal{V}, \mathcal{W}: (\mathbb{R}, \leq) \to \mathbf{Vec}^\ast_{\mathbb{K}}$ are said to be \textit{$\varepsilon$-interleaved} if there exist natural transformations $\phi: \mathcal{V} \Rightarrow \Sigma^\varepsilon \mathcal{W}$ and $\psi: \mathcal{W} \Rightarrow \Sigma^\varepsilon \mathcal{V}$ such that
\[
(\Sigma^\varepsilon \psi) \circ \phi = \Sigma^{2\varepsilon}, \quad (\Sigma^\varepsilon \phi) \circ \psi = \Sigma^{2\varepsilon}.
\]
\end{definition}
The $\varepsilon$-interleaving between $\mathcal{V}$ and $\mathcal{W}$ can also be described by the following commutative diagrams:
\[
\xymatrix@=0.6cm{
  & \Sigma^{\varepsilon} \mathcal{W} \ar[rd]^{\Sigma^{\varepsilon} \psi} & \\
  \mathcal{V} \ar[ru]^{\phi} \ar[rr]^{\Sigma^{2\varepsilon}} && \Sigma^{2\varepsilon} \mathcal{V},
}
\qquad \qquad
\xymatrix@=0.6cm{
  & \Sigma^{\varepsilon} \mathcal{V} \ar[rd]^{\Sigma^{\varepsilon} \phi} & \\
  \mathcal{W} \ar[ru]^{\psi} \ar[rr]^{\Sigma^{2\varepsilon}} && \Sigma^{2\varepsilon} \mathcal{W}.
}
\]

\begin{definition}
The \textit{interleaving distance} between two weighted persistence modules $\mathcal{V}, \mathcal{W}: (\mathbb{R}, \leq) \to \mathbf{Vec}^\ast_{\mathbb{K}}$ is defined as
\begin{equation*}
  d_{I}(\mathcal{V},\mathcal{W})=\inf\{\varepsilon\geq 0\mid \text{$\mathcal{V}$ and $\mathcal{W}$ are $\varepsilon$-interleaved}\}.
\end{equation*}
\end{definition}

Note that if two weighted persistence modules are $0$-interleaved, then we have $\psi \circ \phi = \mathrm{id}_{\mathcal{V}}$ and $\phi \circ \psi = \mathrm{id}_{\mathcal{W}}$, meaning that $ \mathcal{V}$ and $\mathcal{W}$ are isomorphic. However, the converse does not necessarily hold. If $d_{I}(\mathcal{V}, \mathcal{W}) = 0$, we cannot immediately conclude that they are $0$-interleaved. This means that the interleaving distance is not a strict metric but rather an (extended) pseudometric.

Let $\mathcal{V}: (\mathbb{R}, \leq) \to \mathbf{Vec}^\ast_{\mathbb{K}}$ be a weighted persistence module. Then for any $w\in \mathbb{R}_{\geq 0}$, the construction $\mathcal{V}_{w}: (\mathbb{R}, \leq) \to \mathbf{Vec}_{\mathbb{K}}$ given by $r\mapsto \mathcal{V}_{w}(r)=\mathcal{V}(r)_{w}$ is a persistence module.

\begin{lemma}\label{lemma:interleaving_distance}
Let $\mathcal{V}, \mathcal{W}: (\mathbb{R}, \leq) \to \mathbf{Vec}^\ast_{\mathbb{K}}$ be two pointwise finite-dimensional weighted persistence modules. Then we have
\begin{equation*}
  d_{I}(\mathcal{V},\mathcal{W}) = \sup\limits_{w\in \mathbb{R}_{\geq 0}} d_{I}(\mathcal{V}_{w},\mathcal{W}_{w}).
\end{equation*}
\end{lemma}
\begin{proof}
First, assume that $\mathcal{V}$ and $\mathcal{W}$ are $\varepsilon$-interleaved. Then, we have the following commuting diagrams:
\begin{equation*}
\xymatrix@=0.6cm{
  & \Sigma^{\varepsilon} \mathcal{W} \ar[rd]^{\Sigma^{\varepsilon} \psi} & \\
  \mathcal{V} \ar[ru]^{\phi} \ar[rr]^{\Sigma^{2\varepsilon}} && \Sigma^{2\varepsilon} \mathcal{V},
}
\qquad \qquad
\xymatrix@=0.6cm{
  & \Sigma^{\varepsilon} \mathcal{V} \ar[rd]^{\Sigma^{\varepsilon} \phi} & \\
  \mathcal{W} \ar[ru]^{\psi} \ar[rr]^{\Sigma^{2\varepsilon}} && \Sigma^{2\varepsilon} \mathcal{W}.
}
\end{equation*}
These diagrams hold for each weight $w$. In particular, for each $w$, we have the following commuting diagrams:
\begin{equation*}
\xymatrix@=0.6cm{
  & \Sigma^{\varepsilon} \mathcal{W}_w \ar[rd]^{\Sigma^{\varepsilon} \psi_w} & \\
  \mathcal{V}_w \ar[ru]^{\phi_w} \ar[rr]^{\Sigma^{2\varepsilon}} && \Sigma^{2\varepsilon} \mathcal{V}_w,
}
\qquad \qquad
\xymatrix@=0.6cm{
  & \Sigma^{\varepsilon} \mathcal{V}_w \ar[rd]^{\Sigma^{\varepsilon} \phi_w} & \\
  \mathcal{W}_w \ar[ru]^{\psi_w} \ar[rr]^{\Sigma^{2\varepsilon}} && \Sigma^{2\varepsilon} \mathcal{W}_w.
}
\end{equation*}
Here, $\phi_w: \mathcal{V}_w \Rightarrow \Sigma^{\varepsilon} \mathcal{W}_w$ and $\psi_w: \mathcal{W}_w \Rightarrow \Sigma^{\varepsilon} \mathcal{V}_w$ are the natural transformations at each weight $w$.

Thus, for each $w \in \mathbb{R}_{\geq 0}$, $\mathcal{V}_w$ and $\mathcal{W}_w$ are $\varepsilon$-interleaved. Hence, we have the inequality
\[
d_I(\mathcal{V}, \mathcal{W}) \geq d_I(\mathcal{V}_w, \mathcal{W}_w)
\]
for each $w \in \mathbb{R}_{\geq 0}$. Therefore, it follows that
\[
d_I(\mathcal{V}, \mathcal{W}) \geq \sup_{w \in \mathbb{R}_{\geq 0}} d_I(\mathcal{V}_w, \mathcal{W}_w).
\]

On the other hand, assume that for each $w \in \mathbb{R}_{\geq 0}$, $\mathcal{V}_w$ and $\mathcal{W}_w$ are $\varepsilon_w$-interleaved, where $\varepsilon_w = d_I(\mathcal{V}_w, \mathcal{W}_w)$. Let $\varepsilon = \sup_{w} \varepsilon_w$. For any $\delta > 0$, each $\mathcal{V}_w$ and $\mathcal{W}_w$ is $(\varepsilon+\delta)$-interleaved. By combining these interleavings for all $w$, we can construct natural transformations $\phi = \bigoplus_w \phi_w$ and $\psi = \bigoplus_w \psi_w$ that satisfy the $(\varepsilon+\delta)$-interleaving condition for $\mathcal{V}$ and $\mathcal{W}$. Therefore, we obtain the inequality
\[
d_I(\mathcal{V}, \mathcal{W}) \leq \varepsilon+\delta = \sup_{w \in \mathbb{R}_{\geq 0}} d_I(\mathcal{V}_w, \mathcal{W}_w) + \delta.
\]
Since $\delta > 0$ is arbitrary, we have
\[
d_I(\mathcal{V}, \mathcal{W}) \leq \sup_{w \in \mathbb{R}_{\geq 0}} d_I(\mathcal{V}_w, \mathcal{W}_w).
\]

Combining the two inequalities, we conclude that
\[
d_I(\mathcal{V}, \mathcal{W}) = \sup_{w \in \mathbb{R}_{\geq 0}} d_I(\mathcal{V}_w, \mathcal{W}_w),
\]
as desired.
\end{proof}

\begin{theorem}\label{theorem:isometry}
Let $\mathcal{V}$ and $\mathcal{W}$ be two pointwise finite-dimensional weighted persistence modules. Let $\mathcal{B}_{\mathcal{V}}$ and $\mathcal{B}_{\mathcal{W}}$ be the weighted barcodes of $ \mathcal{V}$ and $\mathcal{W}$, respectively. Then we have
\[
d_B(\mathcal{B}_{\mathcal{V}}, \mathcal{B}_{\mathcal{W}}) = d_I(\mathcal{V}, \mathcal{W}).
\]
\end{theorem}

\begin{proof}
For the weighted barcodes $\mathcal{B}_{\mathcal{V}}$ and $\mathcal{B}_{\mathcal{W}}$, a weight-preserving matching $\gamma: \mathcal{B}_{\mathcal{V}} \to \mathcal{B}_{\mathcal{W}}$ can be viewed as a collection of matchings at each weight level, which are then combined into a global matching. Therefore, we can write $\gamma$ as $(\gamma_w)_{w \in \mathbb{R}_{\geq 0}}$, where each $\gamma_w$ is a matching at weight $w$. Let $\mathcal{B}_{\mathcal{V}}^{w}$ and $\mathcal{B}_{\mathcal{W}}^{w}$ denote the collections of bars of $\mathcal{B}_{\mathcal{V}}$ and $\mathcal{B}_{\mathcal{W}}$ with weight $w$, respectively. 

The bottleneck distance for each weight $w$ is given by
\begin{equation*}
  d_B(\mathcal{B}_{\mathcal{V}}^{w}, \mathcal{B}_{\mathcal{W}}^{w}) = \inf_{\gamma_w} \sup_{\substack{([b_1, d_1], w) \in \mathcal{B}_{\mathcal{V}}^{w} \\ \gamma_w(([b_1, d_1], w)) = ([b_2, d_2], w)}} \max\left( |b_1 - b_2|, |d_1 - d_2| \right).
\end{equation*}
Consequently, the global bottleneck distance can be written as
\begin{equation*}
  d_B(\mathcal{B}_{\mathcal{V}}, \mathcal{B}_{\mathcal{W}}) = \sup_w d_B(\mathcal{B}_{\mathcal{V}}^{w}, \mathcal{B}_{\mathcal{W}}^{w}).
\end{equation*}
By the isometry theorem for persistence modules \cite[Theorem 3.5]{bauer2015induced}, for each weight $w$, we have
\begin{equation*}
  d_B(\mathcal{B}_{\mathcal{V}}^{w}, \mathcal{B}_{\mathcal{W}}^{w}) = d_I(\mathcal{V}_w, \mathcal{W}_w).
\end{equation*}
Since the weighted persistence modules $\mathcal{V}$ and $\mathcal{W}$ are pointwise finite-dimensional, the number of non-trivial weights is finite, and thus the supremum over $w$ is actually a maximum. By Lemma \ref{lemma:interleaving_distance}, we have
\begin{equation*}
  d_I(\mathcal{V}, \mathcal{W}) = \sup_w d_I(\mathcal{V}_w, \mathcal{W}_w) 
= \sup_w d_B(\mathcal{B}_{\mathcal{V}}^w, \mathcal{B}_{\mathcal{W}}^w) 
= d_B(\mathcal{B}_{\mathcal{V}}, \mathcal{B}_{\mathcal{W}}).
\end{equation*}
This completes the proof.
\end{proof}

\section{On the stability of magnitude invariants}

In this section, we investigate the stability properties of magnitude-related invariants. Our analysis is divided into two primary perspectives. First, by fixing the point set $X$, we examine the stability of persistent magnitude homology with respect to the choice of the filtration center $z$ and the scaling function $f$. We provide a formal treatment of this stability from a categorical perspective.
Second, we address the challenge of metric perturbations. Since the global structure of a finite metric space is inherently rigid, magnitude homology is known to be sensitive to even minor displacements of the points in $X$. To mitigate this, we turn our study to the stability of magnitude for persistence finite metric spaces, investigating the conditions under which magnitude becomes a robust descriptor even when the underlying point set is subjected to noise or deformation.

\subsection{Stability of persistent magnitude homology}

Let $X$ be a finite point set in Euclidean space, which can naturally be regarded as a finite metric space. For a fixed point $z$, consider its spherical neighborhood
\begin{equation*}
  N_{r}(z)=\{x\in X \mid \|x-z\|\leq r\},\quad r > 0.
\end{equation*}
The point set $N_{r}(z)$ can also be regarded as a finite metric space.

\begin{proposition}
The construction
\[
N(z): (\mathbb{R}, \leq) \to \mathbf{FinPts}^{\hookrightarrow},\quad r\mapsto N_{r}(z)
\]
defines a functor.
\end{proposition}

\begin{proof}
For any real numbers $r_1 \leq r_2$, the map
\[
N_{r_1}(z) \hookrightarrow N_{r_2}(z)
\]
is an embedding of point sets. The functoriality follows by a straightforward verification.
\end{proof}

The above construction provides a filtration of the finite point set $X$, which can also be regarded as a filtration of the corresponding finite metric space. This provides the construction of persistent magnitude homology by allowing one to consider neighborhoods at varying length scales.

Based on the distance-based filtration, one can study the persistent magnitude homology of a finite point set $X$. For each fixed point $z$ and length scale $l \geq 0$, consider the chain complexes $\mathrm{MC}_{*,l}(N_{r}(z))$ as $r$ varies. The corresponding magnitude homology groups $\mathrm{MH}_{k,l}(N_{r}(z))$ then form a persistence module
\[
\mathrm{MH}_{k,l}^{a,b}(N(z)) : \mathrm{MH}_{k,l}(N_{a}(z)) \longrightarrow \mathrm{MH}_{k,l}(N_{b}(z)), \quad a \leq b.
\]
The rank of $\mathrm{MH}_{k,l}^{a,b}(N(z))$ is called the $(a,b)$-persistent magnitude Betti number. Analogous to classical persistent homology, one can construct the barcodes and persistence diagrams associated with these persistence modules, providing a multi-scale summary of the topological and geometric features captured by magnitude homology.

Now, let $f:\mathbb{R}\to \mathbb{R}^{+}$ be a non-decreasing function. The generalized inverse of $f$, denoted by $f^{-1}$, is defined as
\[
f^{-1}(y) = \inf \{ x \in \mathbb{R} \mid f(x) \geq y \}.
\]
This gives the generalized inverse as the smallest $x$ such that $f(x) \geq y$, which is well-defined due to the monotonicity of $f$.

For a given point set $X$ and a fixed point $z$ in the space, we obtain a persistence finite metric space
\[
N(z,f) = N_{f^{-1}(-)}(z): (\mathbb{R}, \leq) \longrightarrow \mathbf{FinMet}_{iso}.
\]
We denote $N_{a}(z,f) = N_{f^{-1}(a)}(z)$. Correspondingly, we have the persistent magnitude homology $\mathrm{MH}_{k,l}^{a,b}(N(z,f))$ for any real numbers $a \leq b$. We denote the associated weighted barcode by $\mathcal{B}(N(z,f))$.

For two non-decreasing functions $f,g:\mathbb{R}\to \mathbb{R}^{+}$, we define their distance
\begin{equation*}
  \|f-g\| = \sup\limits_{x\in \mathbb{R}} |f(x)-g(x)|.
\end{equation*}
Then we have the stability theorem for persistent magnitude homology.

\begin{theorem}\label{theorem:stability_radius}
Let $f,g: \mathbb{R} \to \mathbb{R}^{+}$ be two non-decreasing functions. Then we have
\[
  d_B(\mathcal{B}(N(z,f)), \mathcal{B}(N(z,g))) \leq \| f - g \|.
\]
\end{theorem}

\begin{proof}
Let $\varepsilon = \| f - g \|$. We claim that $f^{-1}(a) \leq g^{-1}(a + \varepsilon)$, which is equivalent to the inclusion $N_{a}(z,f) \hookrightarrow N_{a + \varepsilon}(z,g)$.

Indeed, for any $y < f^{-1}(a)$, by the definition of the generalized inverse, we have $f(y) < a$. Since $\|f - g\| = \varepsilon$, it follows that $g(y) \leq f(y) + \varepsilon < a + \varepsilon$. This implies that no $y < f^{-1}(a)$ can satisfy $g(y) \geq a + \varepsilon$, and thus
\[
g^{-1}(a + \varepsilon) \geq f^{-1}(a).
\]
This establishes the inclusion
\[
  N_{a}(z,f)=N_{f^{-1}(a)}(z) \hookrightarrow N_{g^{-1}(a + \varepsilon)}(z) = \Sigma^{\varepsilon} N(z,g) = N_{a + \varepsilon}(z,g).
\]
By a symmetric argument, we also have the reverse inclusion
\[
  N_{a}(z,g) \hookrightarrow \Sigma^{\varepsilon} N_{a}(z,f).
\]
These inclusions give rise to the following natural transformations
\[
  \phi: N(z,f) \Rightarrow \Sigma^{\varepsilon} N(z,g), \quad \psi: N(z,g) \Rightarrow \Sigma^{\varepsilon} N(z,f).
\]
It is directly verified that
\[
  (\Sigma^{\varepsilon}\psi)\phi = \Sigma^{2\varepsilon} : N(z,f) \to \Sigma^{2\varepsilon} N(z,f),
\]
which effectively shifts $N_{a}(z,f)$ to $N_{a+2\varepsilon}(z,f)$ for any $a \in \mathbb{R}$. Similarly, we also have
\[
  (\Sigma^{\varepsilon}\phi)\psi = \Sigma^{2\varepsilon} : N(z,g) \to \Sigma^{2\varepsilon} N(z,g).
\]
Hence, the persistence finite metric spaces $N(z,f)$ and $N(z,g)$ are $\varepsilon$-interleaved. It follows that
\[
  d_I(N(z,f), N(z,g)) \leq \varepsilon.
\]
Now, consider the composition of functors
\[
\xymatrix{
  (\mathbb{R}, \leq)  \ar@<0.5ex>[rr]^-{N(z,f)} \ar@<-0.5ex>[rr]_-{N(z,g)} && \mathbf{FinMet}_{iso} \ar@{->}[rr]^-{\mathrm{MH}_{\ast,\ast}} && \mathbf{Vec}^{\ast}_{\mathbb{K}}.
}
\]
By \cite[Proposition 3.6]{bubenik2014categorification}, we have the inequality
\[
  d_I(\mathrm{MH}_{\ast,\ast}(N(z,f)), \mathrm{MH}_{\ast,\ast}(N(z,g))) \leq d_I(N(z,f), N(z,g)) \leq \varepsilon.
\]
By Theorem \ref{theorem:isometry}, we obtain
\[
  d_I(\mathrm{MH}_{\ast,\ast}(N(z,f)), \mathrm{MH}_{\ast,\ast}(N(z,g))) = d_B(\mathcal{B}(N(z,f)), \mathcal{B}(N(z,g))).
\]
Thus, we conclude that
\[
  d_B(\mathcal{B}(N(z,f)), \mathcal{B}(N(z,g))) \leq \varepsilon.
\]
This completes the proof.
\end{proof}

\begin{theorem}\label{theorem:stability_center}
Let $f: \mathbb{R} \to \mathbb{R}^{+}$ be a non-decreasing subadditive function, i.e., $f(x+y) \leq f(x) + f(y)$ for all $x, y \geq 0$. Then we have
\[
  d_B(\mathcal{B}(N(z,f)), \mathcal{B}(N(z',f))) \leq f(\| z - z' \|).
\]
\end{theorem}

\begin{proof}
Let $d = \| z - z' \|$ and $\varepsilon = f(d)$. We claim that $f^{-1}(a) + d \leq f^{-1}(a + \varepsilon)$, which implies the inclusion $N_a(z, f) \subseteq N_{a + \varepsilon}(z', f)$.

Indeed, for any $y < f^{-1}(a) + d$, we have $y - d < f^{-1}(a)$, and thus $f(y - d) < a$ by the definition of the generalized inverse. Since $f$ is subadditive, we have
\[
f(y) = f((y - d) + d) \leq f(y - d) + f(d) < a + \varepsilon.
\]
This implies that no $y < f^{-1}(a) + d$ can satisfy $f(y) \geq a + \varepsilon$, and thus
\[
f^{-1}(a + \varepsilon) \geq f^{-1}(a) + d.
\]
This establishes the inclusion
\begin{equation*}
  N_a(z, f) \subseteq N_{a + \varepsilon}(z', f).
\end{equation*}
By a symmetric argument, we also have
\begin{equation*}
  N_a(z', f) \subseteq N_{a + \varepsilon}(z, f).
\end{equation*}
Then, we obtain the following natural transformations
\[
  \phi: N(z, f) \Rightarrow \Sigma^{\varepsilon} N(z', f), \quad \psi: N(z', f) \Rightarrow \Sigma^{\varepsilon} N(z, f),
\]
such that
\begin{equation*}
  (\Sigma^{\varepsilon}\psi) \phi = \Sigma^{2\varepsilon}: N(z, f) \to \Sigma^{2\varepsilon} N(z, f),
\end{equation*}
and
\begin{equation*}
  (\Sigma^{\varepsilon}\phi) \psi = \Sigma^{2\varepsilon}: N(z', f) \to \Sigma^{2\varepsilon} N(z', f).
\end{equation*}
Thus, the persistence finite metric spaces $N(z, f)$ and $N(z', f)$ are $\varepsilon$-interleaved. Therefore, we have
\[
  d_I(N(z, f), N(z', f)) \leq \varepsilon.
\]
The remaining part of the proof follows step-by-step similarly to that of Theorem \ref{theorem:stability_radius}, by applying the composition of functors and the isometry theorem to conclude
\[
  d_B(\mathcal{B}(N(z, f)), \mathcal{B}(N(z', f))) \leq \varepsilon = f(\| z - z' \|).
\]
This completes the proof.
\end{proof}

\begin{remark}
The subadditivity assumption in Theorem \ref{theorem:stability_center} is not overly restrictive and is satisfied by a broad class of functions commonly used in practice. For instance, any concave function $f: \mathbb{R}^{+} \to \mathbb{R}^{+}$ with $f(0) \ge 0$ is subadditive. 
\end{remark}

A direct corollary of Theorem \ref{theorem:stability_center}, when $f = \mathrm{id}$, is as follows.

\begin{corollary}
For two fixed points $z$ and $z'$, we have
\[
  d_B(\mathcal{B}(N(z)), \mathcal{B}(N(z'))) \leq \| z - z' \|.
\]
\end{corollary}

\begin{theorem}\label{theorem:stability_composition}
Let $f,g: \mathbb{R} \to \mathbb{R}^{+}$ be two non-decreasing functions. Assume $f$ is subadditive. Then we have
\[
  d_B(\mathcal{B}(N(z,f)), \mathcal{B}(N(z',g))) \leq \| f - g \|+f(\| z - z' \|).
\]
\end{theorem}

\begin{proof}
By the triangle inequality of the bottleneck distance, we have
\begin{align*}
d_B(\mathcal{B}(N(z,f)), \mathcal{B}(N(z',g))) &\leq d_B(\mathcal{B}(N(z,f)), \mathcal{B}(N(z',f))) + d_B(\mathcal{B}(N(z',f)), \mathcal{B}(N(z',g))) \\
&\leq f(\| z - z' \|) + \| f - g \|,
\end{align*}
where the first term is bounded by Theorem \ref{theorem:stability_center} and the second term is bounded by Theorem \ref{theorem:stability_radius}. This completes the proof.
\end{proof}

\subsection{Stability of magnitude for persistence finite metric space}

Let $X$ be a finite set in a Euclidean space $E$, and let $u_X$ denote the barycenter of $X$. We have the persistence finite metric space associated with $X$ as a filtration
\begin{equation*}
    \mathcal{N}(X): (\mathbb{R}, \leq) \to \mathbf{FinPts}^{\hookrightarrow}
\end{equation*}
given by
\begin{equation*}
    \mathcal{N}_r(X) = \{x \in X \mid \|x - u_X\| \leq r\}.
\end{equation*}

Let $T: E \to E$ be an isometry. Since isometries preserve the relative distances within the point set, it follows that the resulting barcodes are invariant, i.e.,
\begin{equation*}
    \mathcal{B}(\mathcal{N}(TX)) = \mathcal{B}(\mathcal{N}(X)).
\end{equation*}

More generally, consider a transformation $T: E \to E$ that perturbs $X$, yielding the mapped set $TX$. Since translations do not affect the weighted barcodes, we may assume without loss of generality that the perturbation $T$ preserves the barycenter, such that $u_{TX} = u_X$. Consequently, we can consistently assume that the barycenter of $X$ is located at the origin $O$. The set of all $n$-point configurations with the barycenter at the origin defines the \textit{centered configuration space}
\[
    \mathcal{C}_n^0(\mathbb{R}^d) = \left\{ (x_1, \dots, x_n) \in (\mathbb{R}^d)^n \mid \sum_{i=1}^n x_i = 0, x_i \neq x_j \text{ for } i \neq j \right\}.
\]
This space is an open submanifold of the subspace $\{(x_1, \dots, x_n) \in (\mathbb{R}^d)^n \mid \sum_{i=1}^n x_i = 0\} \cong \mathbb{R}^{d(n-1)}$, and thus possesses dimension $d(n-1)$. We equip $(\mathbb{R}^d)^n$ with the standard Euclidean metric
\[
    d((x_1, \dots, x_n), (y_1, \dots, y_n)) = \left( \sum_{i=1}^n \|x_i - y_i\|^2 \right)^{1/2},
\]
where $\|\cdot\|$ denotes the Euclidean norm. By restricting this metric, $\mathcal{C}_n^0(\mathbb{R}^d)$ becomes a well-defined metric space.

For the unordered case, we can consider the quotient space
\begin{equation*}
    \widetilde{\mathcal{C}}_n^0(\mathbb{R}^d) \cong \mathcal{C}_n^0(\mathbb{R}^d) / S_n,
\end{equation*}
where $S_n$ is the symmetric group. For any two unordered point sets $X, Y \in \widetilde{\mathcal{C}}_n^0(\mathbb{R}^d)$, we use their distance via the Wasserstein metric
\begin{equation*}
    d_{W,p}(X, Y) = \inf_{\gamma: X \to Y} \left( \sum_{x \in X} \|x - \gamma(x)\|^p \right)^{1/p},
\end{equation*}
where the infimum is taken over all bijections $\gamma: X \to Y$.

Our primary objective was to establish the stability of persistent magnitude homology with respect to perturbations of $X$. However, this poses significant challenges because magnitude homology inherently encodes fine-grained metric information, making it highly sensitive to geometric noise. The following example illustrates this instability.

\begin{example}
Consider two configurations of three points in Euclidean plane: 
\[
X = \{x_1, x_2, x_3\}, \quad x_1=(0,0),\, x_2=(1,0),\, x_3=(2,0),
\]
\[
Y = \{y_1, y_2, y_3\}, \quad y_1=(0,0),\, y_2=(1,\epsilon),\, y_3=(2,0), \quad \epsilon>0.
\]
Here, $\epsilon$ denotes a small positive perturbation.

For $X$, the pairwise distances are $d_{12}=d_{23}=1$ and $d_{13}=2$. The 1-chains of length $1$ are $(1,2),(2,1),(2,3),(3,2)$, yielding $\mathrm{MH}_{1,1}(X)\cong\mathbb{Z}^4$. For length $2$, the 1-chains are $(1,3),(3,1)$ and the corresponding 2-chains $(1,2,3),(3,2,1)$ satisfy $\partial(1,2,3)=-(1,3)$, $\partial(3,2,1)=-(3,1)$, giving $\mathrm{MH}_{1,2}(X)=0$. Other lengths yield trivial homology, so $\mathrm{MH}_{1,*}(X)\cong \mathbb{Z}^4$ concentrated at $\ell=1$.

For $Y$, let $a=\sqrt{1+\epsilon^2}$, so that $d_{12}=d_{23}=a$ and $d_{13}=2$. The 1-chains of length $a$ are $(1,2),(2,1),(2,3),(3,2)$, with no 2-chains of total length $a$, giving $\mathrm{MH}_{1,a}(Y)\cong\mathbb{Z}^4$. For length $2$, the 1-chains $(1,3),(3,1)$ exist but have no corresponding 2-chains, hence $\mathrm{MH}_{1,2}(Y)\cong\mathbb{Z}^2$. Other lengths yield trivial homology, so $\mathrm{MH}_{1,a}(Y)\cong \mathbb{Z}^4$ and $\mathrm{MH}_{1,2}(Y)\cong \mathbb{Z}^2$.

The 1-dimensional magnitude homology is summarized in the following table:

\begin{center}
\begin{tabular}{c|c|c|c}
\toprule
Length $\ell$ & $1$ & $a$ & $2$ \\
\midrule
$X$ & $\mathbb{Z}^4$ & $0$ & $0$ \\
$Y$ & $0$ & $\mathbb{Z}^4$ & $\mathbb{Z}^2$ \\
\bottomrule
\end{tabular}
\end{center}

A small perturbation $\epsilon>0$ shifts the original length-1 class to $\ell=a$ and introduces a new $\mathbb{Z}^2$ at length 2, demonstrating a discontinuous change in $\mathrm{MH}_{1,*}$ and the sensitivity of magnitude homology to metric perturbations.
\end{example}

In light of these difficulties, we shift our focus to the stability of the magnitude itself. Recall that the magnitude can be interpreted as a function of the scale $r$. For a given point set $X$, let $\magn_X: \mathbb{R} \to \mathbb{R}$ be the function that maps each $r$ to the magnitude of the subset $\mathcal{N}_r(X)$. 

In the following, we first establish an upper bound estimate for the magnitude.

\begin{theorem}\label{theorem:magnitude_bound}
Let $X = \{x_1, \dots, x_n\} \subset \mathbb{R}^d$ be a finite set of $n$ points admitting a nonnegative weighting. Suppose the pairwise distances in $X$ are at most $2L$. Then the magnitude of $X$ satisfies the upper bound
\[
\magn(X) \le \frac{n}{1 + (n-1)e^{-2L}}.
\]
\end{theorem}

\begin{proof}
First, we recall that the Euclidean space $\mathbb{R}^d$ is a metric space of negative type. According to Schoenberg's Theorem, the exponential kernel $k(x, y) = e^{-d(x, y)}$ is strictly positive definite on $\mathbb{R}^d$. Thus, the similarity matrix $Z$ defined by $Z_{ij} = e^{-d(x_i, x_j)}$ is a symmetric positive definite matrix \cite{schoenberg1938metric}.

For any finite metric space with a positive definite similarity matrix, by \cite{meckes2013positive}, the magnitude can be expressed via the variational formula
\begin{equation*}\label{eq:variational}
\mathrm{Mag}(X) = \sup_{u \in \mathbb{R}^n, u \neq 0} \frac{(\sum_{i=1}^n u_i)^2}{\sum_{i=1}^n \sum_{j=1}^n u_i u_j e^{-d(x_i, x_j)}}.
\end{equation*}

Let $Q(u) = \sum_{i,j} u_i u_j e^{-d(x_i, x_j)}$ be the quadratic form in the denominator. Given $x_i, x_j \in X$, the maximum possible distance between any two points is the diameter of the ball, i.e., $d(x_i, x_j) \le 2L$. Since the function $e^{-x}$ is monotonically decreasing, we have
\begin{equation*}
e^{-d(x_i, x_j)} \ge e^{-2L}, \quad \forall i \neq j.
\end{equation*}
Now, consider the denominator $Q(v)$ for $v \ge 0$, we obtain that
\begin{align*}
Q(v) &= \sum_{i=1}^n v_i^2 + \sum_{i \neq j} v_i v_j e^{-d(x_i, x_j)} \\
&\ge \sum_{i=1}^n v_i^2 + e^{-2L} \sum_{i \neq j} v_i v_j \\
&= \sum_{i=1}^n v_i^2 + e^{-2L} \left[ \left( \sum_{i=1}^n v_i \right)^2 - \sum_{i=1}^n v_i^2 \right] \\
&= (1 - e^{-2L}) \sum_{i=1}^n v_i^2 + e^{-2L} \left( \sum_{i=1}^n v_i \right)^2
\end{align*}
By the Cauchy-Schwarz inequality, one has
\begin{equation*}
  \sum_{i=1}^n v_i^2 \ge \frac{1}{n} (\sum_{i=1}^n v_i)^2.
\end{equation*}
It follows that
\begin{align*}
Q(v) &\ge \frac{1 - e^{-2L}}{n} \left( \sum_{i=1}^n v_i \right)^2 + e^{-2L} \left( \sum_{i=1}^n v_i \right)^2 \\
&= \frac{1 - e^{-2L} + n e^{-2L}}{n} \left( \sum_{i=1}^n v_i \right)^2 \\
&= \frac{1 + (n - 1)e^{-2L}}{n} \left( \sum_{i=1}^n v_i \right)^2
\end{align*}
Finally, substituting this lower bound for $Q(v)$ into the variational definition
\begin{equation*}
\mathrm{Mag}(X) \le \frac{(\sum v_i)^2}{\frac{1 + (n - 1)e^{-2L}}{n} (\sum v_i)^2} = \frac{n}{1 + (n - 1)e^{-2L}},
\end{equation*}
as desired.
\end{proof}

\begin{remark}
The condition that $X$ admits a nonnegative weighting is essential for the inequality to hold. If this condition is dropped, the upper bound can be violated, even in Euclidean space.
\end{remark}

\begin{example}
Consider the Euclidean space $\mathbb{R}^4$. Let $X = \{x_1, \dots, x_6\}$ be a finite set of $n=6$ points consisting of the 5 vertices of a regular 4-simplex with edge length $D$ and its barycenter $x_6$. Explicitly, the coordinates of the 5 vertices $x_1, \dots, x_5$ can be chosen as
\begin{align*}
x_1 &= \left( \frac{D\sqrt{10}}{5},\ 0,\ 0,\ 0 \right) \\[6pt]
x_2 &= \left( -\frac{D\sqrt{10}}{20},\ \frac{D\sqrt{6}}{4},\ 0,\ 0 \right) \\[6pt]
x_3 &= \left( -\frac{D\sqrt{10}}{20},\ -\frac{D\sqrt{6}}{12},\ \frac{D\sqrt{3}}{3},\ 0 \right) \\[6pt]
x_4 &= \left( -\frac{D\sqrt{10}}{20},\ -\frac{D\sqrt{6}}{12},\ -\frac{D\sqrt{3}}{6},\ \frac{D}{2} \right) \\[6pt]
x_5 &= \left( -\frac{D\sqrt{10}}{20},\ -\frac{D\sqrt{6}}{12},\ -\frac{D\sqrt{3}}{6},\ -\frac{D}{2} \right)
\end{align*}
and the 6th point is the geometric barycenter $x_6 =  (0,0,0,0)$.

By direct calculation, the distance from the barycenter to any vertex $x_i$ ($1 \le i \le 5$) is
\[
\|x_i\| = \sqrt{\left(\frac{D\sqrt{10}}{5}\right)^2} = \frac{D\sqrt{10}}{5} = \sqrt{\frac{2}{5}}D.
\]
The maximum pairwise distance in $X$ is the edge length of the simplex, $D$. Therefore, we set $2L = D$, which implies $L = D/2$. Let $R = \sqrt{\frac{2}{5}}D$ be the distance from $x_6$ to any vertex.

Define $x = e^{-R}$ and $y = e^{-D} = e^{-2L}$. The similarity matrix $Z$ is defined by $Z_{ij} = e^{-d(x_i, x_j)}$. By symmetry, the entries of $Z$ are
\begin{itemize}
    \item $Z_{66} = 1$
    \item $Z_{i6} = Z_{6i} = x$ for $1 \le i \le 5$
    \item $Z_{ij} = y$ for $1 \le i, j \le 5$ and $i \neq j$
\end{itemize}

Let $w_6$ be the weight of the barycenter $x_6$ and $w_A$ be the weight of each vertex. By the definition of magnitude, the weights satisfy $Z\mathbf{w} = \mathbf{1}$, which yields the system
\begin{align*}
    w_6 + 5x w_A &= 1, \\
    x w_6 + (1 + 4y) w_A &= 1.
\end{align*}
Solving this linear system gives
\[
w_6 = \frac{1+4y-5x}{1+4y-5x^2}, \qquad w_A = \frac{1-x}{1+4y-5x^2}.
\]
The magnitude of $X$ is
\[
\magn(X) = w_6 + 5w_A = \frac{6+4y-10x}{1+4y-5x^2}.
\]
The upper bound proposed in Theorem~\ref{theorem:magnitude_bound} is
\[
\text{Upper Bound} = \frac{n}{1 + (n-1)e^{-2L}} = \frac{6}{1+5y}.
\]
For a concrete numerical verification, let $D = 0.1$, so $L = 0.05$. Then $R = \sqrt{0.4} \times 0.1 \approx 0.0632455$. Thus, the magnitude is
\[
\magn(X) \approx  1.088016,
\]
where upper bound is
\[
\text{Upper Bound} =\frac{6}{1+5e^{-0.1}} \approx 1.086170.
\]
Clearly, $1.088016 > 1.086170$, confirming that the inequality fails when the nonnegative weighting condition is removed.
\end{example}

\begin{remark}
For any finite set $X$ of $n$ points in a positive definite metric space, we have the trivial bounds $1 \le \magn(X) \le n$. In Theorem \ref{theorem:magnitude_bound}, when $L = 0$, we have $\magn(X) = 1$. As $L \to \infty$, the upper bound approaches $n$, which holds even without the nonnegative weighting assumption, consistent with the absolute bounds for $\magn(X)$.

Moreover, it is shown in \cite{meckes2015magnitude} that if a finite set $X \subseteq \mathbb{R}^d$ is contained in a ball of radius $L$, say $X \subseteq B_L$, then $\magn(X) \le \magn(B_L)$. In particular, this provides an upper bound for $\magn(X)$ depending only on $L$.

The magnitude of a convex compact set $A \subset \mathbb{R}^d$ is conjectured to be
\[
\magn(A) = \sum_{k=0}^d \frac{1}{k! \omega_k} V_k(A),
\]
where $V_k(A)$ denotes the $k$-th intrinsic volume of $A$ and $\omega_k$ is the volume of the $k$-dimensional unit ball. In three-dimensional Euclidean space, Barceló and Carbery \cite{barcelo2018magnitudes} proved that the magnitude of the ball of radius $L$ is exactly
\[
\magn(B_L) = 1 + 2L + L^2 + \frac{1}{6}L^3.
\]
However, for higher dimensions ($d \geq 5$), the above conjecture is false \cite{barcelo2018magnitudes}. In general, for a convex body $A \subset \mathbb{R}^d$ (including balls), Meckes \cite{meckes2020magnitude} established an elegant upper bound
\[
\magn(A) \le \sum_{k=0}^{d} \frac{\omega_k}{4^k} V_k(A).
\]
However, estimating the magnitude remains one of the ongoing challenges in the study of magnitude.
\end{remark}

Next, we investigate the sensitivity of $\mathrm{Mag}(X)$ under small perturbations of the points in $X$. It is noteworthy that $\mathrm{Mag}(X)$ becomes highly sensitive to such perturbations as the distance between any two points in $X$ approaches zero. This singularity makes it challenging to maintain the stability of $\mathrm{Mag}(X)$ in a general setting. To address this, we restrict our analysis to point sets where the pairwise distances are bounded below by a threshold $\delta > 0$. This leads us to consider the \textit{thick configuration space}, defined as
\[
\mathcal{C}_n^{\delta}(\mathbb{R}^d) = \left\{ (x_1, \dots, x_n) \in (\mathbb{R}^d)^n \mid \|x_i - x_j\| \ge \delta, \text{ for } i \neq j \right\}.
\]
In some contexts, to eliminate translational invariance, we may further impose a barycentric constraint
\[
\mathcal{C}_n^{\delta}(\mathbb{R}^d) = \left\{ (x_1, \dots, x_n) \in (\mathbb{R}^d)^n \mid \sum_{i=1}^n x_i = 0, \|x_i - x_j\| \ge \delta, \text{ for } i \neq j \right\}.
\]
We denote this space as $\mathcal{C}_n^{\delta}(\mathbb{R}^d)$ in the subsequent stability analysis.

\begin{theorem}\label{theorem:diiference_bound}
Let $X, Y \in \mathcal{C}_n^{\delta}(\mathbb{R}^d)$ be two finite sets of $n$ points such that $\max_i \|x_i - y_i\|_2 \le \epsilon$. The variation in their magnitude satisfies
\begin{equation*}
|\mathrm{Mag}(X) - \mathrm{Mag}(Y)| \le C_{n,d} \frac{\epsilon}{\delta},
\end{equation*}
where $C_{n,d}$ is a constant depending on the number of points $n$ and the dimension $d$.
\end{theorem}

\begin{proof}
Let $Z_X$ and $Z_Y$ be the interpolation matrices for the point sets $X$ and $Y$ respectively, using the exponential kernel $e^{-\|x_i - x_j\|}$. The Magnitude is defined as $\mathrm{Mag}(X) = \mathbf{1}^T Z_X^{-1} \mathbf{1}$. Using the resolvent identity $A^{-1} - B^{-1} = A^{-1}(B - A)B^{-1}$, we can write
\begin{equation*}
\begin{aligned}
\mathrm{Mag}(Y) - \mathrm{Mag}(X) &= \mathbf{1}^T (Z_Y^{-1} - Z_X^{-1}) \mathbf{1} \\
&= \mathbf{1}^T Z_Y^{-1} (Z_X - Z_Y) Z_X^{-1} \mathbf{1} \\
&= w_Y^T (Z_X - Z_Y) w_X,
\end{aligned} \label{eq:resolvent_mag}
\end{equation*}
where $w_X = Z_X^{-1} \mathbf{1}$ and $w_Y = Z_Y^{-1} \mathbf{1}$ are the corresponding weight vectors.

To begin with, we estimate the spectral norm of the perturbation matrix $E = Z_X - Z_Y$. Given that the exponential kernel $f(r) = e^{-r}$ is 1-Lipschitz on $[0, \infty)$, combining this property with the reverse triangle inequality yields
\[
|E_{ij}| = |e^{-\|x_i - x_j\|} - e^{-\|y_i - y_j\|}| \le |\|x_i - x_j\| - \|y_i - y_j\|| \le \|(x_i - y_i) - (x_j - y_j)\|.
\]
By the triangle inequality and the condition $\max_i \|x_i - y_i\| \le \epsilon$, we have $|E_{ij}| \le \|x_i - y_i\| + \|x_j - y_j\| \le 2\epsilon$. Consequently, as $E$ is a symmetric matrix, its spectral norm is bounded by its maximum absolute row sum, leading to
\begin{equation*}
\|Z_X - Z_Y\|_2 \le \|E\|_\infty = \max_{1 \le i \le n} \sum_{j=1}^n |E_{ij}| \le 2n\epsilon. \label{eq:matrix_bound}
\end{equation*}

Since $Z_X^{-1}$ is symmetric positive definite, we use its symmetric square root $Z_X^{-1/2}$ to decompose the squared $\ell^2$-norm
\[
\begin{aligned}
\|w_X\|_2^2 &= \|Z_X^{-1} \mathbf{1}\|_2^2 = \|Z_X^{-1/2} (Z_X^{-1/2} \mathbf{1})\|_2^2 \\
&\le \|Z_X^{-1/2}\|_2^2 \cdot \|Z_X^{-1/2} \mathbf{1}\|_2^2.
\end{aligned}
\]
By the spectral mapping theorem, $\|Z_X^{-1/2}\|_2^2 = \|Z_X^{-1}\|_2$. The second term is exactly the magnitude $\|Z_{X}^{-1/2} \mathbf{1}\|_2^2 = \mathbf{1}^T Z_X^{-1} \mathbf{1} = \mathrm{Mag}(X)$. Therefore
\begin{equation*}
\|w_X\|_2^2 \le \|Z_X^{-1}\|_2 \cdot \mathrm{Mag}(X). \label{eq:weight_bound}
\end{equation*}
The upper bound for the norm of the inverse matrix $\|Z_X^{-1}\|_2$ is established in \cite{narcowich1992norm}. Specifically, following the refined results in Chapter 12 of \cite{wendland2004scattered}, we have
\begin{equation*}
  \|Z_X^{-1}\|_2 \le C_d / \delta.
\end{equation*}
Combined with the fact that $\mathrm{Mag}(X) \le n$, we obtain
\[
\|w_X\|_2 \le \sqrt{\frac{C_d n}{\delta}}.
\]

Finally, applying the Cauchy-Schwarz inequality to the expression for the magnitude difference
\begin{equation*}
\begin{aligned}
|\mathrm{Mag}(Y) - \mathrm{Mag}(X)| &= |w_Y^T (Z_X - Z_Y) w_X| \\
&\le \|w_Y\|_2 \cdot \|Z_X - Z_Y\|_2 \cdot \|w_X\|_2 \\
&\le \sqrt{\frac{C_d n}{\delta}} \cdot (2n\epsilon) \cdot \sqrt{\frac{C_d n}{\delta}} \\
&= \frac{2n^2 C_d \epsilon}{\delta}.
\end{aligned}
\end{equation*}
By setting $C_{n,d} = 2n^2 C_d$, the proof is complete.
\end{proof}

Next, we investigate the stability of the magnitude for persistent finite metric spaces. For a point set $X$, let $\mathrm{Mag}_X: \mathbb{R} \to \mathbb{R}$ be the function given by
\begin{equation*}
    \mathrm{Mag}_X(r) = \mathrm{Mag}(\mathcal{N}_r(X)),
\end{equation*}
where $\mathcal{N}_r(X)$ is the persistent finite metric space as defined previously. This function, often referred to as the \textit{magnitude profile} of $X$, captures the geometric evolution of the point set across different scales $r$.

To quantify the similarity between the magnitude profiles of two point sets $X$ and $Y$, we employ the truncated $L^1$ distance. For two integrable functions $f, g \in L^1([0, L])$, their distance is defined as
\begin{equation*}
    d_{L^1}(f, g) = \int_{0}^{L} |f(r) - g(r)| \, dr ,
\end{equation*}
where $L \in (0, \infty]$ is a fixed constant representing the maximum observation scale. In the context of our stability analysis, we consider $f = \mathrm{Mag}_X$ and $g = \mathrm{Mag}_Y$.

\begin{theorem}\label{theorem:stability}
Let $X, Y \in \mathcal{C}_n^{\delta}(\mathbb{R}^d)$ be two finite point sets contained within a disk of radius $L$. Then the distance between their magnitude profiles satisfies
\begin{equation*}
    d_{L^1}(\mathrm{Mag}_X, \mathrm{Mag}_Y)  \leq K_{n,d,L,\delta} \, d_{W,\infty}(X, Y),
\end{equation*}
where $K_{n,d,L,\delta}$ is a constant depending on $n, d, L, \delta$.
\end{theorem}

\begin{proof}
Let $\epsilon = d_{W,\infty}(X, Y)$. By the definition of the $\infty$-Wasserstein distance, there exists a bijection between $X$ and $Y$ such that the maximum displacement is $\epsilon$. Without loss of generality, we index the points as $X = \{x_1, \dots, x_n\}$ and $Y = \{y_1, \dots, y_n\}$ such that $\|x_i - y_i\| \le \epsilon$ for all $i$. 
Since both sets are centered at the origin ($u_X = u_Y = 0$), let $r_i = \|x_i\|$ and $r'_i = \|y_i\|$. By the reverse triangle inequality, we have 
\begin{equation*}
    |r_i - r'_i| \le \|x_i - y_i\| \le \epsilon \quad \text{for all } i=1, \dots, n.
\end{equation*}

For any scale $r \in [0, L]$, the persistent subsets are $\mathcal{N}_r(X) = \{x_i \mid r_i \le r\}$ and $\mathcal{N}_r(Y) = \{y_i \mid r'_i \le r\}$. We partition the interval $[0, L]$ into a \textit{matching region} $\Omega$ and its complement $\Omega^c = [0, L] \setminus \Omega$. Specifically, $r \in \Omega$ if and only if for every $i$, $r_i \le r \Leftrightarrow r'_i \le r$. In $\Omega$, the subsets $\mathcal{N}_r(X)$ and $\mathcal{N}_r(Y)$ contain points with exactly the same indices, meaning they are perfectly paired under the bijection.

First, we bound the measure of the mismatching region $\Omega^c$. If $r \in \Omega^c$, there must exist some index $i$ such that $r \in [\min(r_i, r'_i), \max(r_i, r'_i))$. Therefore, $\Omega^c$ is contained in the union of these intervals:
\begin{equation*}
    |\Omega^c| \le \sum_{i=1}^n |r_i - r'_i| \le n\epsilon.
\end{equation*}

We now decompose the $L^1$ distance between the magnitude profiles:
\begin{align*}
   & \int_{0}^{L} |\mathrm{Mag}_X(r) - \mathrm{Mag}_Y(r)| \, dr \\[6pt]
    =& \int_{\Omega} |\mathrm{Mag}(\mathcal{N}_r(X)) - \mathrm{Mag}(\mathcal{N}_r(Y))| \, dr + \int_{\Omega^c} |\mathrm{Mag}_X(r) - \mathrm{Mag}_Y(r)| \, dr.
\end{align*}

For the first integral over the matching region $\Omega$, the points in $\mathcal{N}_r(X)$ and $\mathcal{N}_r(Y)$ are perfectly paired with displacement at most $\epsilon$, and the minimum separation distance is at least $\delta$. Applying Theorem \ref{theorem:difference_bound}, the integrand is uniformly bounded by
\[
\sup_{r \in \Omega} |\mathrm{Mag}(\mathcal{N}_r(X)) - \mathrm{Mag}(\mathcal{N}_r(Y))| \le \frac{C_{n,d} \epsilon}{\delta}.
\]
Integrating this constant over $\Omega \subseteq [0, L]$, we have
\begin{equation*}
    \int_{\Omega} |\mathrm{Mag}(\mathcal{N}_r(X)) - \mathrm{Mag}(\mathcal{N}_r(Y))| \, dr \le L \cdot \frac{C_{n,d} \epsilon}{\delta}.
\end{equation*}

For the second integral over the mismatching region $\Omega^c$, we use the trivial bound for magnitude difference. Since $\mathrm{Mag}(\mathcal{N}_r(X)) \le n$, we have $|\mathrm{Mag}_X(r) - \mathrm{Mag}_Y(r)| \le n$. Using the measure bound $|\Omega^c| \le n\epsilon$, we obtain
\begin{equation*}
    \int_{\Omega^c} |\mathrm{Mag}_X(r) - \mathrm{Mag}_Y(r)| \, dr \le |\Omega^c| \cdot n \le (n\epsilon) \cdot n = n^2 \epsilon.
\end{equation*}

Combining the above estimates, we get
\begin{equation*}
    \int_{0}^{L} |\mathrm{Mag}_X(r) - \mathrm{Mag}_Y(r)| \, dr  \leq L \cdot \frac{C_{n,d} \epsilon}{\delta} + n^2 \epsilon = \left( \frac{C_{n,d} L}{\delta} + n^2 \right) \epsilon.
\end{equation*}
By setting $K_{n,d,L,\delta} = \frac{C_{n,d} L}{\delta} + n^2$, we conclude that $d_{L^1}(\mathrm{Mag}_X, \mathrm{Mag}_Y) \leq K_{n,d,L,\delta} \, \epsilon$.
\end{proof}

It is important to observe that, although Theorem \ref{theorem:stability} provides an explicit error bound for the perturbation of the magnitude profile, the stability constant $K_{n,d,L,\delta}$ diverges as the minimum separation distance $\delta$ tends to zero. This implies that the stability of the magnitude profile is inherently conditional, contingent upon the geometric separation of the point set. Such a phenomenon reflects the transition from a well-posed regime to a singular state: as points approach collision ($\delta \to 0$), the interpolation matrix becomes increasingly ill-conditioned, leading to a loss of Lipschitz continuity in the magnitude profile. Consequently, in practical applications involving noisy data, a sufficient separation distance must be maintained to ensure the robustness of the magnitude as a geometric descriptor.

\section*{Acknowledgements}
We thank Mark W. Meckes and his PhD student Yitong Zheng for their valuable suggestions and support in this paper.

\section*{Funding information}

This work was are supported by the Scientific Research Foundation of Chongqing University of Technology, the High-Level Scientific Research Foundation of Hebei Province, the Start-up Research Fund of BIMSA and a grant from SIMIS (with project titled ``New type of topology and applicaitions based on GLMY theory", 2025), Beijing Natural Science Foundation (International Scientists, Project Grant No.S25081), National Natural Science Foundation of China (NSFC) (General Program, Grant No.82573048), and Start-up Research Fund of University of North Carolina at Charlotte.

\section*{Conflict of interest}

The authors declare no competing interests.

\bibliographystyle{plain}
\bibliography{Reference}

\begin{thebibliography}{10}

\bibitem{asao2023magnitude}
Yasuhiko Asao.
\newblock Magnitude homology and path homology.
\newblock {\em Bulletin of the London Mathematical Society}, 55(1):375--398,
  2023.

\bibitem{asao2024girth}
Yasuhiko Asao, Yasuaki Hiraoka, and Shu Kanazawa.
\newblock Girth, magnitude homology and phase transition of diagonality.
\newblock {\em Proceedings of the Royal Society of Edinburgh Section A:
  Mathematics}, 154(1):221--247, 2024.

\bibitem{barcelo2018magnitudes}
Juan~Antonio Barcel{\'o} and Anthony Carbery.
\newblock On the magnitudes of compact sets in {E}uclidean spaces.
\newblock {\em American Journal of Mathematics}, 140(2):449--494, 2018.

\bibitem{bauer2015induced}
Ulrich Bauer and Michael Lesnick.
\newblock Induced matchings and the algebraic stability of persistence
  barcodes.
\newblock {\em Journal of Computational Geometry}, 6(2):162--191, 2015.

\bibitem{bi2025topological}
Wanying Bi, Hongsong Feng, Jie Wu, Jingyan Li, and Guo-Wei Wei.
\newblock Topological magnitude for protein flexibility analysis.
\newblock {\em Royal Society Open Science}, 12(12), 2025.

\bibitem{bi2024persistent}
Wanying Bi, Xin Fu, Jingyan Li, and Jie Wu.
\newblock Persistent magnitude for the quantitative analysis of the structure
  and stability of carboranes.
\newblock {\em J. Comput. Biophys. Chem}, 1(11):11, 2024.

\bibitem{bi2024magnitude}
Wanying Bi, Jingyan Li, and Jie Wu.
\newblock The magnitude homology of a hypergraph.
\newblock {\em Homology, Homotopy and Applications}, 26(2):325--348, 2024.

\bibitem{bubenik2014categorification}
Peter Bubenik and Jonathan~A Scott.
\newblock Categorification of persistent homology.
\newblock {\em Discrete \& Computational Geometry}, 51(3):600--627, 2014.

\bibitem{chazal2009proximity}
Fr{\'e}d{\'e}ric Chazal, David Cohen-Steiner, Marc Glisse, Leonidas~J Guibas,
  and Steve~Y Oudot.
\newblock Proximity of persistence modules and their diagrams.
\newblock In {\em Proceedings of the twenty-fifth annual symposium on
  Computational geometry}, pages 237--246, 2009.

\bibitem{edelsbrunner2008persistent}
Herbert Edelsbrunner, John Harer, et~al.
\newblock Persistent homology-a survey.
\newblock {\em Contemporary mathematics}, 453(26):257--282, 2008.

\bibitem{govc2021persistent}
Dejan Govc and Richard Hepworth.
\newblock Persistent magnitude.
\newblock {\em Journal of Pure and Applied Algebra}, 225(3):106517, 2021.

\bibitem{hepworth2022magnitude}
Richard Hepworth.
\newblock Magnitude cohomology.
\newblock {\em Mathematische Zeitschrift}, 301(4):3617--3640, 2022.

\bibitem{hepworth2017categorifying}
Richard Hepworth and Simon Willerton.
\newblock Categorifying the magnitude of a graph.
\newblock {\em Homology, Homotopy and Applications}, 19(2):31--60, 2017.

\bibitem{kaneta2021magnitude}
Ryuki Kaneta and Masahiko Yoshinaga.
\newblock Magnitude homology of metric spaces and order complexes.
\newblock {\em Bulletin of the London Mathematical Society}, 53(3):893--905,
  2021.

\bibitem{leinster2013magnitude}
Tom Leinster.
\newblock The magnitude of metric spaces.
\newblock {\em Documenta Mathematica}, 18:857--905, 2013.

\bibitem{leinster2019magnitude}
Tom Leinster.
\newblock The magnitude of a graph.
\newblock In {\em Mathematical Proceedings of the Cambridge Philosophical
  Society}, volume 166, pages 247--264. Cambridge University Press, 2019.

\bibitem{leinster2017magnitude}
Tom Leinster and Mark~W Meckes.
\newblock The magnitude of a metric space: from category theory to geometric
  measure theory.
\newblock {\em Measure Theory in Non-Smooth Spaces}, pages 156--193, 2017.

\bibitem{leinster2021magnitude}
Tom Leinster and Michael Shulman.
\newblock Magnitude homology of enriched categories and metric spaces.
\newblock {\em Algebraic \& Geometric Topology}, 21(5):2175--2221, 2021.

\bibitem{leinster2013asymptotic}
Tom Leinster and Simon Willerton.
\newblock On the asymptotic magnitude of subsets of {E}uclidean space.
\newblock {\em Geometriae Dedicata}, 164(1):287--310, 2013.

\bibitem{meckes2013positive}
Mark~W Meckes.
\newblock Positive definite metric spaces.
\newblock {\em Positivity}, 17(3):733--757, 2013.

\bibitem{meckes2015magnitude}
Mark~W Meckes.
\newblock Magnitude, diversity, capacities, and dimensions of metric spaces.
\newblock {\em Potential Analysis}, 42(2):549--572, 2015.

\bibitem{meckes2020magnitude}
Mark~W Meckes.
\newblock On the magnitude and intrinsic volumes of a convex body in
  {E}uclidean space.
\newblock {\em Mathematika}, 66(2):343--355, 2020.

\bibitem{narcowich1992norm}
Francis~J Narcowich and Joseph~D Ward.
\newblock Norm estimates for the inverses of a general class of scattered-data
  radial-function interpolation matrices.
\newblock {\em Journal of Approximation Theory}, 69(1):84--109, 1992.

\bibitem{otter2022magnitude}
Nina Otter.
\newblock Magnitude meets persistence: homology theories for filtered
  simplicial sets.
\newblock {\em Homology, homotopy and applications}, 24(2):365--387, 2022.

\bibitem{sazdanovic2021torsion}
Radmila Sazdanovic and Victor Summers.
\newblock Torsion in the magnitude homology of graphs.
\newblock {\em Journal of Homotopy and Related Structures}, 16(2):275--296,
  2021.

\bibitem{schoenberg1938metric}
Isaac~J Schoenberg.
\newblock Metric spaces and positive definite functions.
\newblock {\em Transactions of the American Mathematical Society},
  44(3):522--536, 1938.

\bibitem{wendland2004scattered}
Holger Wendland.
\newblock {\em Scattered data approximation}, volume~17.
\newblock Cambridge university press, 2004.

\bibitem{zomorodian2004computing}
Afra Zomorodian and Gunnar Carlsson.
\newblock Computing persistent homology.
\newblock In {\em Proceedings of the twentieth annual symposium on
  Computational geometry}, pages 347--356, 2004.

\end{thebibliography}

\end{document}